\documentclass{amsart}

\usepackage[normalem]{ulem}

\usepackage{lineno}

\usepackage{soul}

\usepackage{amssymb, amsmath, amsthm, hyperref, euscript, color}
\usepackage[dvipsnames]{xcolor}

\usepackage{amscd, enumerate}

\usepackage{bbm}

\usepackage[english]{babel}

\theoremstyle{plain}

\newtheorem{theorem}{Theorem}

\newtheorem{corollary}{Corollary}

\newtheorem{proposition}{Proposition}

\newtheorem{remark}{Remark}

\newtheorem{lemma}{Lemma}

\newtheorem{thm}{Theorem}

\newtheorem{cor}[thm]{Corollary}

\theoremstyle{definition}

\newtheorem{definition}{Definition}

\newcommand{\R}{\mathbb{R}}

\newcommand{\Z}{\mathbb{Z}}

\newcommand{\N}{\mathbb{N}}

\newcommand{\C}{\mathbb{C}}

\newcommand{\re}{\textnormal{Re}}

\DeclareMathOperator{\So}{SO}
\DeclareMathOperator{\U}{U}
\DeclareMathOperator{\Dim}{dim}
\DeclareMathOperator{\Id}{id}
\DeclareMathOperator{\Gl}{GL}
\DeclareMathOperator{\Bu}{BU}
\DeclareMathOperator{\Bs}{BSO}

\DeclareMathOperator{\Ke}{Ker}
\DeclareMathOperator{\Mod}{mod}

\begin{document}

\author{Rafael Torres}
\address{Scuola Internazionale Superiori di Studi Avanzati (SISSA)\\ Via Bonomea 265\\34136\\Trieste\\Italy}
\email{rtorres@sissa.it}

\title[Pseudo-holomorphic and CR embeddings]{An equivalence between pseudo-holomorphic embeddings into almost-complex Euclidean space and CR regular embeddings into complex space}



\maketitle

\emph{Abstract}: We show that a pseudo-holomorphic embedding of an almost-complex $2n$-manifold into almost-complex $(2n + 2)$-Euclidean space exists if and only if there is a CR regular embedding of the $2n$-manifold into complex $(n + 1)$-space. We remark that the fundamental group does not place any restriction on the existence of either kind of embedding when $n$ is at least three. We give necessary and sufficient conditions in terms of characteristic classes for a closed almost-complex 6-manifold to admit a pseudo-holomorphic embedding into $\R^8$ equipped with an almost-complex structure that need not be integrable.  

\section{Introduction and main results}

In this paper, we study existence of pseudo-holomorphic embeddings of almost-complex manifolds into almost-complex Euclidean space and CR regular embeddings into complex space. We begin with the discussion concerning the almost-complex realm. Let $M$ be a closed smooth manifold whose tangent bundle $TM$ admits a complex structure, i.e., an automorphism $J_M:TM\rightarrow TM$ such that $J_M^2 = - \Id$. The automorphism $J_M$ is called an almost-complex structure and the pair $(M, J_M)$ is called an almost-complex manifold. 

\begin{definition}\label{Definition PSEmbedding} Let $(M, J_M)$ and $(N, J_N)$ be almost-complex manifolds. A smooth embedding\begin{equation}\label{Mapf}f:(M, J_M)\hookrightarrow (N, J_N)\end{equation} is a pseudo-holomorphic embedding if and only if\begin{equation}Tf\circ J_M = J_N\circ Tf.\end{equation}That is, the embedding $f$ in (\ref{Mapf}) is a pseudo-holomorphic embedding if and only if the tangent map is complex linear at each point.\end{definition}

Let $(\R^{2m}, \hat{J})$ be Euclidean space of dimension $2m$ equipped with an almost-complex structure $\hat{J}$. The maximum principle says that the canonical integrable almost-complex structure  $\R^{2m}\cong \C^m$ does not contain compact complex submanifolds of positive dimension. Besides being non-canonical, the almost-complex structure $\hat{J}$ need not be integrable in the study of maps as in Definition \ref{Definition PSEmbedding}. Topological obstructions for the existence of pseudo-holomorphic embeddings of almost-complex 2m-manifolds into almost-complex Euclidean (4m + 2)-space were studied by Di Scala-Kasuya-Zuddas in \cite{[DKZ]}. The authors also studied  conditions for the existence of a pseudo-holomorphic embedding $(M, J_M)\hookrightarrow (\R^6, \hat{J})$ for an almost-complex manifold of (real) dimension four \cite[Theorem 5]{[DKZ]}. Their arguments can be modified by allowing $J_M$ to be changed in order to show that such an embedding exists if and only if the 4-manifold is parallelizable.
The constraint on the tangent bundle of a manifold that admits a codimension two pseudo-holomorphic embedding into an almost-complex Euclidean space in fact holds for arbitrary dimension.

\begin{thm}\label{Theorem ParaPS} If there are almost-complex structures $(M, J_M)$ and $(\R^{2n + 2}, \hat{J})$ such that there is a pseudo-holomorphic embedding $f:(M, J_M)\hookrightarrow (\R^{2n + 2}, \hat{J})$, then the $2n$-manifold $M$ is parallelizable. 

Let $M$ be a parallelizable 2n-manifold that smoothly embeds into $\R^{2n + 2}$. There exist almost-complex structures $(M, J_M)$ and $(\R^{2n + 2}, \hat{J})$ for which there is a pseudo-holomorphic embedding $f:(M, J_M)\hookrightarrow (\R^{2n + 2}, \hat{J})$.
\end{thm}

Necessary and sufficient conditions for the existence of a pseudo-holomorphic embedding of 6-manifolds are as follows. The first Pontrjagin class is denoted by $p_1$, and the third Chern class by $c_3$.

\begin{thm}\label{Theorem 6D} Let $M$ be a closed smooth simply connected 6-manifold with torsion-free homology $H_\ast(M)$ and second Stiefel-Whitney class $w_2(M) = 0$. Let $(M, J)$ be a given almost-complex structure and assume $c_1(M, J) = 0$. There is a pseudo-holomorphic embedding\begin{equation}(M, J)\hookrightarrow (\R^8, \hat{J})\end{equation}for some almost-complex structure $\hat{J}$ on $\R^8$ if and only if\begin{equation}c_3(M) = 0 = p_1(M).\end{equation}
\end{thm}

A closed smooth orientable 6-manifold $M$ with torsion-free homology admits an almost-complex structure if and only if the image of $w_2(M)$ under the Bockstein homomorphism $\beta: H^2(M; \Z/2)\rightarrow H^3(M; \Z)$ vanishes, i.e., $\beta w_2(M) = 0$. The latter is equivalent to $w_2(M) = c_1(L)$ mod 2 for the first Chern class of a complex line bundle $L$ over $M$ \cite{[W]}.  

Kotschick has shown that any finitely presented group is the fundamental group of an almost-complex manifold of dimension greater or equal to four \cite{[Kotschick]}. Examples of pseudo-holomorphic embeddings of high-dimensional manifolds with prescribed fundamental group are given in our next result. 

\begin{thm}\label{Theorem G} Let $G$ be a finitely presented group and let $n\geq 3$. There exists a closed smooth almost-complex 2n-manifold $(M(G, 2n), J)$ with fundamental group $\pi_1(M(G, 2n)) \cong G$ and for which there exists a pseudo-holomorphic embedding\begin{equation}(M(G, 2n), J)\hookrightarrow (\R^{2n + 2}, \hat{J})\end{equation}for some almost-complex structure $\hat{J}$ on $\R^{2n + 2}$.
\end{thm}

Theorem \ref{Theorem G} could be compared to the situation in dimension four, where the possible choices of fundamental group of a parallelizable 4-manifold are heavily restricted.

Let us now turn our attention to CR regular embeddings into complex space. For this, we now consider a closed smooth real 2n-manifold $M$ of real dimension $\Dim_{\R}(M) = 2n$, a complex manifold $(X, J)$ of complex dimension $\Dim_{\C}(X) = n + 1$, a smooth embedding\begin{equation}\label{GenericEmbedding}f: M\hookrightarrow X,\end{equation} and the bundle\begin{equation}f_{\ast}T_pM\cap J f_{\ast}T_pM\subset T_{f(p)}X.\end{equation}A point $p\in M$ is said to be CR regular provided that \begin{equation}\label{CmplxDim}\Dim_{\C}(f_{\ast}T_pM\cap J f_{\ast}T_pM) = n - 1.\end{equation}The points of $M$ whose complex tangent space has complex dimension equal to n are called complex or CR singular (see Section \ref{Section CmplxPts}, \cite{[S1], [S2]} and the references there for further details). 

\begin{definition}\label{Definition CR embeddings} An embedding $f: M\hookrightarrow X$ for which every point $p\in M$ is CR regular is said to be a CR regular embedding. 
\end{definition}

In \cite{[S1], [S2], [S3]}, Slapar determined topological obstructions to be able to deform a generic smooth embedding as in (\ref{GenericEmbedding}) into a CR regular embedding (see Section \ref{Section CmplxPts}). He showed that a 4-manifold admits a CR regular embedding into $\C^3$ if and only if the 4-manifold is parallelizable \cite{[S2]}. Necessary and sufficient conditions for the existence of CR regular embeddings of 6-manifolds into $\C^4$ were studied by the author of this note in \cite{[T]}. Building on work of Kervaire \cite{[Kervaire2]}, we extend these results to arbitrary even dimensions.

\begin{cor}\label{Theorem CRParallelizable} Let $n\in \N$ and suppose $M$ is a closed smooth real manifold of real dimension $\Dim_{\R}(M) = 2n$ that can be smoothly embedded into $\R^{2n + 2}$. There is a CR regular embedding\begin{equation}f:M\hookrightarrow \C^{n + 1}\end{equation}if and only if $M$ is parallelizable. 
\end{cor}

Coupling Corollary \ref{Theorem CRParallelizable} with Theorem \ref{Theorem ParaPS} yields the following equivalence between the two kinds of embeddings that are under study in this paper.

\begin{cor}\label{Theorem Equivalence} A closed smooth real manifold $M$ of real dimension $\Dim_{\R}(M) = 2n$ admits a CR regular embedding\begin{equation}M\hookrightarrow \C^{n + 1}\end{equation} if and only if the are almost-complex structures $(M, J_M)$ and $(\R^{2n + 2}, \hat{J})$ for which there is a pseudo-holomorphic embedding\begin{equation}(M, J)\hookrightarrow (\R^{2n + 2}, \hat{J}).\end{equation}
\end{cor}

Finally, we present the following myriad of examples of CR regular embeddings into high-dimensional complex spaces.

\begin{cor}\label{Theorem 1} For any finitely presented group $G$ and integer number $n\geq 3$, there exists a closed smooth real orientable manifold $M(G, 2n)$ of real dimension $\Dim_\R (M(G, 2n)) = 2n$ whose fundamental group is isomorphic to $G$ and for which there is a CR regular embedding\begin{equation}M(G, 2n)\hookrightarrow \C^{n + 1}.\end{equation}
\end{cor}

We end the introduction with a blueprint of the paper. Section \ref{Section TrivialTM} collects the results of Kervaire that we use through out the paper to determine when a manifold is parallelizable, as well as the construction of the manifold used in the proofs of Theorem \ref{Theorem G} and Corollary \ref{Theorem 1}. Almost-complex structures on Euclidean spaces and pseudo-holomorphic embeddings are discussed in Section \ref{Section AlmCmplxStr}, which includes the proof of Theorem \ref{Theorem ParaPS}. The results on CR regular embeddings that were mentioned in the introduction are described at further length in Section \ref{Section CmplxPts}. This section also contains the main ingredients for our proof of Corollary \ref{Theorem CRParallelizable}, which is given in Section \ref{Section ProofCRParallelizable}. The proof of Corollary \ref{Theorem Equivalence} follows from the results that are proven in Sections \ref{Section AlmCmplxStr} and \ref{Section CmplxPts} as it is mentioned in Section \ref{Section ProofTheoremEquivalence}. The proof of Theorem \ref{Theorem 6D} is given in Section \ref{Section Proof6D}. Proofs of Theorem \ref{Theorem G} and Corollary \ref{Theorem 1} are given in Section \ref{Section ProofG}.

\subsection{Acknowledgements:} We thank Daniele Zuddas for useful e-mail exchanges. We thank the referee for her/his detailed report on the initial version of this note, where he/she pointed out several inaccuracies, mistakes, and necessary revisions. Her/his input and insightful remarks were of great value for the improvement of the note.

\section{Parallelizable closed $n$-manifolds with arbitrary fundamental group and that embedded in $\R^{n + 2}$ for $n\geq 5$}\label{Section TrivialTM} Recall that a smooth $n$-manifold $M$ of real dimension $n$ is called parallelizable if its tangent bundle $TM$ is trivial. If the Whitney sum with a trivial bundle $TM\oplus \epsilon$ is trivial, then $M$ is said to be stably parallelizable \cite[Section 3]{[KervaireMilnor]}. We begin this section by describing the topological obstruction to trivializing the tangent bundle of a manifold building on work of Kervaire \cite{[Kervaire1], [Kervaire2]} (cf. \cite[\textsection1]{[JW]}). Associated to the stabilization map $\pi_{n - 1}(\So(n))\overset{\sigma}{\longrightarrow} \pi_{n - 1}(\So)$ there are groups\begin{equation}\mathcal{K}_n:=\Ke(\pi_{n - 1}(\So(n))\overset{\sigma}{\longrightarrow} \pi_{n - 1}(\So)),\end{equation}and the obstruction for a stably parallelizable manifold $M$ to be parallelizable is an element\begin{equation}\psi(M)\in H^n(M; \mathcal{K}_n)\cong \mathcal{K}_n.\end{equation}

The groups $\mathcal{K}_n$ were computed by Kervaire in \cite{[Kervaire1]}. Stably parallelizable manifolds of dimension one, three or seven are parallelizable and $\mathcal{K}_1 = \mathcal{K}_3 = \mathcal{K}_7 = 0$. For all other odd values of $n$, $\mathcal{K}_n = \Z/2$ and Kervaire showed that $\psi(M)$ vanishes if and only if the Kervaire semi-characteristic $\hat{\chi}_{\Z/2}(M)$ is zero. Recall that the Kervaire semi-characteristic of a closed smooth $n$-manifold $M$ of dimension $n = 2k + 1$ is defined as\begin{equation}\label{SemiChar}\hat{\chi}_{\Z/2}(M):= \sum^k_{i = 0}\Dim H^i(M; \Z/2) \Mod 2\end{equation} for $k\in \N$. When the dimension $n$ is even, $\mathcal{K}_n = \Z$ and the element $\psi(M)$ can be identified with the Euler characteristic $\chi(M)$. The work of Kervaire yields the following criteria to decide when a manifold is parallelizable.

\begin{proposition}\label{Proposition Kervaire} Kervaire \cite{[Kervaire1], [Kervaire2]}.

$(i)$ Let $M$ be a closed smooth stably parallelizable $n$-manifold and suppose that the dimension $n$ is odd and $n\neq 1, 3, 7$. The tangent bundle $TM$ is trivial if and only if $\hat{\chi}_{\Z/2}(M) = 0$.

$(ii)$ The tangent bundle $TM$ of a closed smooth stably parallelizable $2n$-manifold $M$ is trivial if and only if $\chi(M) = 0$. In particular, a closed smooth orientable 2n-manifold that embeds into $\R^{2n + 2}$ is parallelizable if and only if its Euler characteristic is zero.

\end{proposition}

The reader is referred to \cite{[Kervaire1], [Kervaire2]} for the proof of Item (i). Due to its role in our proofs of the main results and for the sake of making the paper as self-contained as possible, we give the proof of the last claim of Item (ii) of Proposition \ref{Proposition Kervaire} following \cite[Section 9]{[Kervaire2]} (cf. \cite[Section 6]{[DKZ]}).

\begin{proof} Let $f$ be an embedding of $M$ into $\R^{2n + 2}$. Since $f$ is of codimension two, the normal bundle $\nu_f(M)$ is trivial \cite[Chapter VIII. Theorem 2]{[Kirby]}, and a trivialization yields the generalized Gauss map $G: M\rightarrow V_2(\R^{2n + 2})$, where $V_2(\R^{2n + 2})$ is the Stiefel manifold of linear injective maps $\R^2\rightarrow \R^{2n + 2}$. The map $G$ is nullhomotopic if and only if the generalized curvatura integra $G_{\ast}([M])\in H_{2n}(V_2(\R^{2n + 2}))$ is zero. Kervaire \cite[Section 9]{[Kervaire2]} expressed $G_{\ast}([M])$ in terms of the Hopf invariant of $M$, its Euler characteristic $\chi(M)$, and its semi-characteristic. Given the existence of the codimension two embedding $f$, the only obstruction is then $\chi(M)$ and since we assumed this is zero, the map $G$ is homotopic to a constant map. Consider the canonical identification $\R^{2n + 2}\cong \C^{n + 1}$ along with standard complex coordinates $(z_1, \ldots, z_{n + 1})$ on the latter. We can now take the holomorphic vector field $\partial/\partial_{z_{n + 1}}$ as a constant map and the tangent bundle $TM$ is homotopic to the pullback $f^{\ast}(T\R^{2n + 2}) = f^{\ast}(\C^{n + 1}) = T\C^{n + 1}|_M = M\times \C^{n + 1} \cong M\times \R^{2n + 2}$ as subbundles of $T\C^{n + 1}\cong T\R^{2n + 2}$. This implies that $TM$ is trivial.
\end{proof}

\begin{remark} An anonymous referee kindly pointed out the following simplification of the previous proof. Instead of considering the identification $\R^{2n + 2} \cong \C^{n + 1}$ as in our argument, the conclusion follows by taking the trace of the nullhomotopy of the generalized Gauss map $G$.\end{remark}

A classical result of Dehn \cite{[Dehn]} states that for any group $G$ that has a finite presentation, there exists a closed smooth stably parallelizable $(n - 1)$-manifold whose fundamental group is isomorphic to $G$. In this section we observe that tweaks to his construction yield the following result. 

\begin{proposition}\label{Proposition G} Let $G$ be a finitely presented group and suppose $n\geq 5$. There exists a closed smooth n-manifold $M(G, n)$ such that

$(i)$ the fundamental group is $\pi_1(M(G, n)) \cong G$,

$(ii)$ the tangent bundle $TM(G, n)$ is trivial, and

$(iii)$ there is an embedding\begin{center}\label{EmbeddingProp}$M(G, n)\hookrightarrow \R^{n + 1}.$\end{center}
\end{proposition}

A closed parallelizable $n$-manifold with prescribed fundamental group is immediately obtained by thickening a finite CW complex with fundamental group $G$ and zero Euler characteristic that is embedded in $\R^N$ for large $N$. Items $(i)$ and $(ii)$ of Proposition \ref{Proposition G} have been proven by Johnson-Walton \cite[Theorem A]{[JW]}. Our proof relies on a simpler argument than theirs and it recovers and strengthens their main result \cite[Theorem A]{[JW]}. 

\begin{proof}Consider the presentation\begin{equation}\label{Group}G = \langle g_1, \ldots, g_s | r_1, \ldots, r_t\rangle \end{equation}that consists of s generators $\{g_j\}$ and t relations $\{r_j\}$. The first step is to construct a closed smooth orientable $(n - 1)$-manifold $X(s)$ with free fundamental group $\langle g_1, \ldots, g_s\rangle = \Z\ast \cdots \ast \Z$ as the connected sum\begin{equation}\label{ConnectedSum}X(s):= S^1\times S^{n - 2}\#\cdots \# S^1\times S^{n - 2}\end{equation} of s copies of the product of the circle with the $(n - 2)$-sphere. Using a general position argument and the fact that an embedded loop in $X(s)$ has codimension at least three, represent the relations $\{r_1, \ldots, r_t\}$ in the presentation (\ref{Group}) by disjoint embedded loops $\gamma_j \subset X(s)$. The tubular neighborhood of each loop $\nu(\gamma_j)$ is diffeomorphic to $S^1\times D^{n - 2}$ for every $j$. Construct the closed smooth manifold\begin{equation}\label{n1manifold}X(G, n - 1):= \Big(X(s)\backslash \overset{t}{\underset{j = 1}{\bigsqcup}} \nu(\gamma_j)\Big)\bigcup \Big(\overset{t}{\underset{j = 1}{\bigsqcup}} (D_j^2\times S^{n - 3})\Big).\end{equation}

The Seifert-van Kampen theorem implies that the manifold $X(G, n - 1)$ has fundamental group isomorphic to (\ref{Group}). The bundle $TX(s)\oplus \epsilon$ is trivial since the connected sum of two stably parallelizable manifolds is stably parallelizable (cf. \cite[Lemma 3.5]{[KervaireMilnor]}). It follows that the manifold $X(G, n - 1)$ is stably parallelizable by \cite[Lemma 5.4]{[KervaireMilnor]}. The argument up to this point is due to Dehn \cite{[Dehn]}, which concludes the proof of Item $(i)$. As we have mentioned before, Johnson-Walton present an argument in \cite[Theorem A]{[JW]} to conclude that items $(ii)$ holds. Our argument is as follows. Use the $(n - 1)$-manifold of (\ref{n1manifold}) to build the $n$-manifold\begin{equation}\label{PreManifold}X(G, n) = ((X(G, n - 1)\times S^1)\backslash (D^{n - 1}\times S^1)) \cup (S^{n - 2}\times D^2),\end{equation}and construct the manifold\begin{equation}\label{GManifold}M(G, n):= X(G, n)\#S^3\times S^{n - 3}.\end{equation}The fundamental group of the manifolds in (\ref{PreManifold}) and (\ref{GManifold}) is isomorphic to $G$, and both manifolds are stably parallelizable. If $n$ is even, then the Euler characteristic of $X(G, n)$ is equal to two and $M(G)$ has zero Euler characteristic. Item (ii) of Proposition \ref{Proposition Kervaire} implies that $M(G, n)$ is parallelizable. Item (i) of Proposition \ref{Proposition Kervaire} allows us to conclude that the claim holds in odd-dimensions as well; notice that that $X(G, 7)$ is parallelizable. This finishes the proof of the claims in Item (i) and Item (ii). There are embeddings
\begin{equation}\label{ConsEmb}M(G, n)\hookrightarrow \R^{n + 2}\end{equation}and\begin{equation}\label{Op1}X(G, n - 1)\hookrightarrow \R^{n}\end{equation}by construction given that there exists an embedding $S^{m_1}\times S^{m_2}\hookrightarrow \R^{m_1 + m_2 + 1}$ for $m_1, m_2 \in \N$. We will abuse notation in these last lines and the hypothesis on the dimension of the manifolds is taken to be $n - 1\geq 5$. If the dimension $(n - 1)$ is odd, then either $X(G, n - 1)$ or\begin{equation}\label{Op2}X(G, n - 1)\# (S^3\times S^{n - 4})\# (S^3\times S^{n - 4})\end{equation} has trivial tangent bundle given that\begin{equation}\hat{\chi}_{\Z/2}(S^2\times S^{n - 3}\# S^2\times S^{n - 3}) = 1 \Mod 2\end{equation} and Proposition \ref{Proposition Kervaire} applies. In either case, the parallelizable manifold (\ref{n1manifold}) or (\ref{Op2}) can be embedded into $\R^n$. If the dimension $(n - 1)$ is even and $X(G, n - 1)$ does not have trivial tangent bundle, there are natural numbers $r_1$ and $r_2$ such that the connected sum\begin{equation}\label{Op3}X(G, n - 1)\# (r_1 - 1)(S^2\times S^{n - 3}) \#( r_2 - 1)(S^3\times S^{n - 4})\end{equation}has zero Euler characteristic and trivial tangent bundle by  Proposition \ref{Proposition Kervaire}. The manifold (\ref{Op3}) embeds into $\R^n$. This concludes the  proof of the claim of Item (iii). Relabeling the manifolds now yields a proof of the proposition. 
\end{proof}

\section{Almost-complex structures on $\R^{2m}$ and some pseudo-holomorphic embeddings}\label{Section AlmCmplxStr} 

A complex structure $J_M:TM\rightarrow TM$ induces a preferred orientation on the manifold $M$. If $M$ is an oriented manifold and the orientations coincide, we say that $J_M$ is positive; otherwise, we say that $J_M$ is negative. The space of positive linear complex structures on Euclidean space $\R^{2n}$\begin{equation}\widetilde{\Gamma}(n) = \Gl^+(\R^{2n})/\Gl(\C^n)\end{equation} consists of matrices that are conjugate to\[J_n = \underset{n}\bigoplus \left( \begin{array}{cc}
0 & - 1  \\
1 & 0 \end{array}\right)\]
by an element $\alpha\in \Gl^+(\R^{2n})$. An almost-complex structure $(\R^{2n}, J_n)$ is a map\begin{equation}J: \R^{2n}\rightarrow \widetilde{\Gamma}(n).\end{equation}

\begin{lemma}\label{Lemma Nullhomotopy} Let $M$ be a closed smooth manifold that embeds into $\R^{2n}$, and let $J:M\rightarrow \widetilde{\Gamma}(n)$ be a smooth map. There is a smooth extension $\widetilde{J}:\R^{2n}\rightarrow \widetilde{\Gamma}(n)$ of $J$ if and only if $J$ is homotopic to a constant map.
\end{lemma}

A proof of Lemma \ref{Lemma Nullhomotopy} can be found in \cite[Section 2]{[DZ1]}. There is a homotopy equivalence between $\widetilde{\Gamma}(n)$ and the homogeneous space\begin{equation}\Gamma(n) = \So(2n)/\U(n),\end{equation}and they share the same homotopy groups. The homotopy groups $\pi_k(\Gamma(n))$ in the range $k\leq 2n - 2$ are called stable homotopy groups, and are given by
\label{HomotopyGroups}\[ \pi_k(\Gamma(n)) \cong \pi_{k + 1}(\So(2n)) =
  \begin{cases}
    0       & \quad \text{if } k = 1, 3, 4, 5\\
    \Z  & \quad \text{if } k = 2, 6\\
    \Z/2  & \quad \text{if } k = 0, 7\\
  \end{cases}
\Mod 8\]
as computed by Bott in \cite{[Bott]}.

We now begin our discussion of pseudo-holomorphic embeddings. We point out the following constraint on the tangent bundle of a smooth manifold that admits a codimension two pseudo-holomorphic embedding into Euclidean space and prove Theorem \ref{Theorem ParaPS}. 

\begin{theorem}\label{Theorem PHP} If there are almost-complex structures $(M, J_M)$ and $(\R^{2n + 2}, \hat{J})$ such that there is a pseudo-holomorphic embedding $f:(M, J_M)\hookrightarrow (\R^{2n + 2}, \hat{J})$, then the $2n$-manifold $M$ is parallelizable. 

Let $M$ be a parallelizable 2n-manifold that smoothly embeds into $\R^{2n + 2}$. There exist almost-complex structures $(M, J_M)$ and $(\R^{2n + 2}, \hat{J})$ for which there is a pseudo-holomorphic embedding $f:(M, J_M)\hookrightarrow (\R^{2n + 2}, \hat{J})$.
\end{theorem}

A parallelizable $2n$-manifold $M$ admits an almost-complex structure $J_M$ since there is a basis for $TM$.


\begin{proof} Suppose there are almost-complex structures on $M$ and $\R^{2n + 2}$ such that the pseudo-holomorphic embedding $f$ exists. The claim follows from Kervaire's Proposition \ref{Proposition Kervaire} once we have shown that $\chi(M) = 0$. To see that the Euler characteristic of the manifold vanishes, we make use of the property that the normal bundle $\nu_f(M)$ is a trivial complex line bundle, and observe that the claim follows from the following standard argument. The identity $\chi(M) = \langle c_n(M), [M]\rangle$ suggests us to argue that the n-th Chern class $c_n(M) = c_n(TM)$ vanishes. Since the pullback $f^{\ast}(T\R^{2n + 2}) = T\R^{2n + 2}|_M = TM\oplus \nu_f(M)$ is the trivial bundle, its Chern class satisfies $c(T\R^{2n + 2}|_M) = 1$ and the Whitney product formula implies $c_n(M) = 0$, as claimed \cite{[MS]}. 

Assume now that there is an embedding $M\hookrightarrow \R^{2n + 2}$ and that the bundle $TM$ is trivial. We argue that there is a nullhomotopic map $g: M\rightarrow \widetilde{\Gamma}(n + 1)$ and apply Lemma \ref{Lemma Nullhomotopy} to conclude the proof of the claim. To construct the map $g$, we equip the bundle $f^{\ast}(T\R^{2n + 2}) = T\R^{2n + 2}|_M$ with the following choices of complex structures. As it was done in the proof of Proposition \ref{Proposition Kervaire}, consider $M$ as a real submanifold of $\C^{n + 1} \cong \R^{2n + 2}$. Since $\chi(M) = 0$, the curvature integra $G_\ast([M])$ vanishes and the generalized Gauss map $G$ is nullhomotopic. This implies that there is a vector bundle monomorphism that identifies $TM$ with $M\times \C^n$ and $\nu_f$ with the complementary $M\times \C$. The tangent bundle $TM$ is homotopy equivalent to $M\times \C^n\subset T\C^{n + 1}|_M = M\times \C^{n + 1}$, and the identification induces an almost-complex structure $J_M$ on $M$ (cf. \cite[Section 6]{[DKZ]}). The normal bundle $\nu_f(M)$ is homotopic to a trivial complex line bundle, and there is an analogous complex structure $J_{\nu}$ on it. This yields a complex structure $(T\R^{2n + 2}|_M, J_M\oplus J_{\nu})$, and hence a map $g:M\rightarrow \widetilde{\Gamma}(n + 1)$ that is nullhomotopic by our choices of complex structures. Lemma \ref{Lemma Nullhomotopy} implies that there is a smooth extension $\widetilde{g}: \R^{2n + 2}\rightarrow \widetilde{\Gamma}(n + 1)$, and hence the pseudo-holomorphic embedding $f$ exists.

\end{proof}

Di Scala-Vezzoni \cite[Theorem 1.2]{[DVz]} have shown that the torus $T^{2n}$ admits a pseudo-holomorphic embedding into $\R^{4n}$ for certain choices of almost-complex structures on the manifolds. Theorem \ref{Theorem PHP} yields an improvement in the sense that it minimizes the codimension of the pseudo-holomorphic embedding.

\begin{corollary} For every almost-complex torus $(T^{2n}, J)$ for which the tangent bundle is trivial as a complex bundle, there is an almost-complex structure $(\R^{2n + 2}, \hat{J})$ and a pseudo-holomorphic embedding $(T^{2n}, J)\hookrightarrow (\R^{2n + 2}, \hat{J})$ for $n\in \N$.
\end{corollary}

\section{CR regular embeddings}\label{Section CmplxPts} Let $M$ be a closed smooth real oriented $2n$-manifold. A way to obtain a CR regular embedding into $\C^{n + 1}$ is to start with a generic smooth embedding\begin{equation}\label{CREmb}f:M\hookrightarrow \R^{2n + 2}\cong \C^{n + 1}\end{equation} that may have complex or CR singular points, and then perturb it into a CR regular one. The complex points of a generic embedding are isolated by Thom transversality theorem \cite{[Thom]}; see \cite[Definition 2.2]{[L]} for a rigorous definition of our use of the adjective 'generic'. Slapar \cite{[S1]} studied necessary and sufficient conditions for such perturbation to exist. We describe the scenario beginning by telling complex points $p\in M$ apart as follows. If the orientation of $T_pM$ agrees with the induced orientation of $T_pM\subset TX$ as a complex subspace, then $p$ is positive. Otherwise, the complex point $p$ is negative. Next consider coordinates $(z, \omega)\in \C^n\times \C$, and $n$ by $n$ matrices $A, B$ with complex entries and such that $B = B^T$. An appropriate choice of coordinates $(z, w)$ and Taylor expansion of $f$ allows for a local expression at $p\in M$ as\begin{equation}w = \overline{z}^{T}Az + \re(z^TBz) + o(|z|^2).\end{equation}Complex points can be then classified in terms of the sign of the determinant of the associated matrix\[ \left( \begin{array}{cc}
A & \overline{B}  \\
B & \overline{A} \end{array}\right).\]

The corresponding complex point is said to be elliptic if the determinant is positive. If the determinant is negative, the complex point is said to be hyperbolic. The reader is directed toward \cite{[S1], [S2]} and the references there for details.

Denote by $e_{\pm}(M)$ the number of positive/negative elliptic complex points and by $h_{\pm}(M)$ the number of positive/negative hyperbolic complex points on $M$.
The Lai indices \cite{[L]} are defined as\begin{equation}I_{\pm}(M): = e_{\pm}(M) - h_{\pm}(M)\end{equation}and can be expressed in terms of characteristic classes these indices as the formula\begin{equation}\label{Equation LI}2I_{\pm}(M) = \chi(M) + \big \langle \overset{n}{\underset{k = 0}\sum} (\pm1)^{k + 1} e^k(\nu(M)) \cup  c_{n - k}(TX|_{M}), [M] \big \rangle,\end{equation}where $\nu(M)$ stands for the normal bundle of $M\rightarrow X$, $e$ and $c_{n - k}$ are the Euler and (n - k)th Chern classes, respectively. These indices are invariant under regular isotopies of embeddings and Slapar has shown that they are the only topological invariants of complex points up to isotopy \cite[Corollary 1.2]{[S1]} and \cite[Theorem 1]{[S3]}. Their vanishing $I_{\pm}(M) = 0$ is a necessary condition for the existence of a regular homotopy between $f$ and a CR regular embedding. 

\begin{proposition}\label{Proposition CRParallelizable} If there is a CR regular embedding\begin{equation}\label{Hypothesis CR}f: M\hookrightarrow \C^{n + 1}\end{equation}for a closed smooth real 2n-manifold $M$, then the tangent bundle $TM$ is trivial. 
\end{proposition}

\begin{proof} The existence of the CR regular embedding (\ref{Hypothesis CR}) implies that the Lai indices (\ref{Equation LI}) vanish, and in particular $\chi(M) = 0$. The claim now follows from the second item of Proposition \ref{Proposition Kervaire}.

\end{proof}

The converse statement is the main ingredient in the proof of Corollary \ref{Theorem 1} and we now state it; its proof follows immediately from the Cancellation theorem \cite[Corollary 1.2]{[S1]} (cf. \cite[Proposition 4]{[S2]}).

\begin{proposition}\label{Proposition Slapar}Slapar \cite{[S1], [S2]}. Let $M$ be a closed smooth real and oriented 2n-manifold, and let $X$ be a complex manifold of $\Dim_{\C}(X) = n + 1$ equipped with a Riemannian metric $h$. Suppose $f:M\rightarrow X$ is a smooth generic embedding, and $\epsilon > 0$. If $I_{\pm}(M) = 0$, then there is a regular isotopy\begin{equation}f_t: M\rightarrow X\end{equation}for $t\in [0, 1]$ that satisfies\begin{enumerate}[(i)]\item $f_0 = f$
\item $h(f_t(p), f(p))< \epsilon$ for every $t\in [0, 1]$ and every point $p\in M$, and\item $f_1:M\rightarrow X$ is a CR regular embedding.\end{enumerate}
\end{proposition}

Slapar's result has the following consequence.

\begin{corollary}\label{Corollary CRG} Let $M$ be a closed smooth real parallelizable 2n-manifold that admits an embedding into $\R^{2n + 2}$. There is a CR regular embedding $M\hookrightarrow \C^{n + 1}$.
\end{corollary}

\begin{proof} Since $TM$ and $\nu_f(M)$ are trivial bundles, the Lai indices (\ref{Equation LI}) are zero. The claim now follows from Proposition \ref{Proposition Slapar}. 
\end{proof}

\begin{remark}\emph{If we require the map that $f$ of Definition \ref{Definition CR embeddings} to be an immersion, we obtain the concept of a \emph{CR regular immersion $f:M\rightarrow \C^{n + 1}$.} Let $M$ be a closed smooth real parallelizable $2n$-manifold. There is a CR regular immersion of $M$ into $\C^{n + 1}$. If $M$ is parallelizable, then it immerses in $\R^{2n + 2}$ \cite[Corollary 3.3]{[Adachi]}. Proposition \ref{Proposition Slapar} can be adapted to immersions so that there is a regular homotopy between $f'$ and a CR regular immersion $f$ provided that the Lai indices vanish \cite[Proposition and Remark 2.1]{[S2]}.}
\end{remark}

\section{Proofs}

\subsection{Proof Theorem \ref{Theorem 6D}}\label{Section Proof6D} It was shown in the proof of Theorem \ref{Theorem PHP} that the existence of a pseudo-holomorphic embedding $f: (M, J)\hookrightarrow (\R^8, \hat{J})$ of a closed smooth orientable 6-manifold implies that the tangent bundle $TM$ is trivial and hence $c_3(M, J) = 0 = p_1(M)$. We now prove the converse using an argument found in \cite[Section 5]{[DKZ]}. The following result of C. T. C. Wall implies the existence of an embedding $f: M\hookrightarrow \R^8$, and hence the normal bundle $\nu_f(M)$ is trivial.

\begin{theorem}\label{Theorem Wall} \cite[Section 9]{[W]}. Suppose $M$ is a closed smooth simply connected 6-manifold with torsion-free homology and $w_2(M) = 0$. There is an embedding\begin{equation}M\hookrightarrow \R^8\end{equation} if and only if $p_1(M) = 0$.
\end{theorem}

Let $J'$ be a complex structure on $\nu_f(M)$ that is compatible with the normal orientation induced by the embedding $f$. We obtain a complex structure\begin{equation}(T\R^8|_M, J\oplus J') \cong (TM\oplus \nu_f(M), J\oplus J').\end{equation} and a map $g: M\rightarrow \widetilde{\Gamma}(4)$ as it was discussed in Section \ref{Section AlmCmplxStr}. To be able to invoke Lemma \ref{Lemma Nullhomotopy} and hence prove the existence of the required almost-complex structure, we claim that $g$ is nullhomotopic. Denote by $M^{(i)}$ the $i$-skeleton of $M$ and consider its cell-decomposition\begin{equation}
M^{(0)}\subset M^{(1)}\subset M^{(2)}\subset M^{(3)}\subset M^{(4)}\subset M^{(5)}\subset M^{(6)}\subset M.\end{equation}

The only nontrivial homotopy groups $\pi_k(\widetilde{\Gamma}(4))$ in the range $k\in \{0, 1, 2, 3, 4, 5, 6\}$ are $k = 2, 6$ by (\ref{HomotopyGroups}), and they both are infinite cyclic. Hence, there are two obstructions\begin{equation}\Omega_2(g)\in H^2(M; \pi_2(\Gamma(4))) = H^2(M; \Z)\end{equation}and\begin{equation}\Omega_6(g)\in H^6(M; \pi_6(\Gamma(4))) = H^6(M; \Z)\end{equation} for $g$ to be nullhomotopic. The map $g$ is nullhomotopic over the 2-skeleton $M^{(2)}$ if and only if $\Omega_2(g) = 0$. If this holds, the homotopy to a constant map can be extended over $M^{(i)}$ for $i\in \{3, 4, 5\}$ since $\pi_i(\widetilde{\Gamma}(4)) = 0$. The extension of the homotopy can be further extended over the 6-skeleton $M^{(6)}$ if and only if $\Omega_6(g) = 0$. We now show that both obstructions vanish and we start with the argument for $\Omega_2(g)$, which appears in \cite[Proof of Theorem 5]{[DKZ]}. There is a fibration of classifying spaces \cite[Chapter 1\textsection 2.J]{[Adachi]}, \cite[Appendix A.2]{[Whitehead]}\begin{equation}\label{FibrationCSpaces}\Gamma(4)\overset{\imath}{\hookrightarrow} \Bu(4)\rightarrow \Bs(8)\end{equation} arising from the inclusion $\U(4)\rightarrow \So(8)$ and
the classifying map of the bundle $TM\oplus \nu_f(M)$ is\begin{equation}\imath\circ g: M\rightarrow \Bu(4).\end{equation}

The part of the homotopy exact sequence of (\ref{FibrationCSpaces}) to be considered is\begin{equation}\pi_3(\Bs(8))\rightarrow \pi_2(\Gamma(4))\rightarrow \pi_2(\Bu(4))\rightarrow \pi_2(\Bs(8))\rightarrow \pi_1(\Gamma(4)),\end{equation}which is $0\rightarrow\Z\rightarrow \Z\rightarrow \Z/2\rightarrow 0$ since $\pi_k(BG) = \pi_{k - 1}(G)$. This implies that the induced map $\imath_{\ast}: \pi_2(\Gamma(4))\rightarrow \pi_2(BU(4))$ is the double map in $\Z$. The restriction $(\imath\circ g)|_{M^{(2)}}$ pins down a class $(\imath\circ g)^\ast c_1$ where $c_1\in H^2(\Bu(4)) = \Z$ is a generator. In particular, $2\Omega_2(g) = c_1(TM\oplus \nu_f(M)) = c_1((M, J\oplus J')) = c_1(M, J)$ (cf. \cite[Lemma 15]{[DKZ]}). The absence of torsion and the hypothesis $c_1(M, J)= 0$ imply $\Omega_2(g) = 0$. 

We now adapt the previous argument to prove $\Omega_6(g) = 0$ as follows. Consider the part of homotopy exact sequence of (\ref{FibrationCSpaces}) given by\begin{equation}\pi_7(\Bs(8))\rightarrow \pi_6(\Gamma(4))\rightarrow \pi_6(\Bu(4))\rightarrow \pi_6(\Bs(8))\rightarrow \pi_5(\Gamma(4)).\end{equation}This sequence reduces to $0\rightarrow \Z\rightarrow \Z\rightarrow 0$. The restriction $(\imath\circ g)|_{M^{(6)}}$ pins down a class $(\imath\circ g)^\ast c_3$ where $c_3\in H^6(\Bu(4)) = \Z$ is a generator. In particular, the obstruction $\Omega_6(g)$ vanishes given that the Euler characteristic of $M$ is zero. This implies that the map $g:M\rightarrow \widetilde{\Gamma}(4)$ is homotopic to a constant map, and hence there is an extension $\widetilde{J}: \R^8\rightarrow \widetilde{\Gamma}(4)$ by Lemma \ref{Lemma Nullhomotopy}.
\hfill $\square$

\subsection{Proof of Corollary \ref{Theorem CRParallelizable}}\label{Section ProofCRParallelizable} Corollary \ref{Corollary CRG} and Proposition \ref{Proposition CRParallelizable} imply the claim.
\hfill $\square$

\subsection{Proof of Corollary \ref{Theorem Equivalence}}\label{Section ProofTheoremEquivalence} The result follows from Theorem \ref{Theorem PHP}, Corollary \ref{Corollary CRG}, and Proposition \ref{Proposition CRParallelizable}. 
\hfill $\square$

\subsection{Proofs of Theorem \ref{Theorem G} and Corollary \ref{Theorem 1}}\label{Section ProofG} Proposition \ref{Proposition G} says that there is an embedding $f:M(G, 2n)\hookrightarrow \R^{2n + 2}$, where the tangent bundle $TM(G, 2n)$ is trivial and the fundamental group $\pi_1(M(G, 2n))$ is isomorphic to $G$. Theorem \ref{Theorem G} follows from Theorem \ref{Theorem PHP}. Considering the identification $\R^{2n + 2}\cong \C^{n + 1}$, Corollary \ref{Theorem 1} follows from Corollary \ref{Theorem CRParallelizable}.

\hfill $\square$


\begin{thebibliography}{99}
\bibitem{[Adachi]} M. Adachi, \emph{Embeddings and immersions}, Translated from the 1984 Japanese original by Kiki Hudson. Translations of Mathematical Monographs, 124. Amer. Math. Soc., Providence, RI, 1993. x + 183pp.

\bibitem{[Bott]} R. Bott, \emph{The stable homotopy of the classical groups}, Ann. of Math. 70 (1959), 313 - 337.

\bibitem{[Dehn]} M. Dehn, \emph{\"Uber unendliche diskontinuierliche Gruppen}, Math. Annalen 71 (1912), 16 - 144.

\bibitem{[DKZ]} A. J. Di Scala, N. Kasuya, and D. Zuddas, \emph{On embeddings of almost complex manifolds in almost complex Euclidean spaces}, J. Geom. Phys. 101 (2016), 19 - 26.

\bibitem{[DVz]} A. J. Di Scala and L. Vezzoni, \emph{Complex submanifolds of almost complex Euclidean spaces}, Q. J. Math. 61 (2010), 401 - 405.

\bibitem{[DZ1]} A. J. Di Scala and D. Zuddas, \emph{Embedding almost-complex manifolds in almost-complex euclidean spaces}, J. Geom. Phys. 61 (2011), 1928 - 1931.


\bibitem{[JW]} F. E. A. Johnson and J. P. Walton, \emph{Parallelizable manifolds and the fundamental group}, Mathematika 47 (2000), 165 - 172.

\bibitem{[Kervaire1]} M. A. Kervaire, \emph{Relative characteristic classes}, Amer. J. Math. 79 (1957), 517 - 558.

\bibitem{[Kervaire2]} M. A. Kervaire, \emph{Some non-stable homotopy groups of Lie groups}, Illinois J. Math. 4 (1960), 161 - 169.

\bibitem{[KervaireMilnor]} M. A. Kervaire and J. Milnor, \emph{Groups of homotopy spheres: I}, Ann. of Math. 77 (1963), 504 - 537. 

\bibitem{[Kirby]} R. C. Kirby, \emph{The topology of 4-manifolds}, Springer Lecture Notes in Mathematics 1374, Springer-Verlag, 1989.


\bibitem{[Kotschick]} D. Kotschick, \emph{All fundamental groups are almost complex}, Bull. London Math. Soc. 24 (1992), 377 - 378.

\bibitem{[L]} H.-F. Lai, \emph{Characteristic classes of real manifolds immersed in complex manifolds}, Trans. Amer. Math. Soc. 172 (1972), 1 - 33.

\bibitem{[MS]} J. W. Milnor and J. D. Stasheff, \emph{Characteristic classes}, Ann. Math. Studies 76, Princeton University Press, Princeton, NJ, 1974 v + 524pp.


\bibitem{[S1]} M. Slapar, \emph{Canceling complex points in codimension 2}, Bull. Aust. Math. Soc. 88 (2013), 64 - 69.

\bibitem{[S2]} M. Slapar, \emph{CR regular embeddings and immersions of compact orientable 4-manifolds into $\C^3$}, Int. J. Math. 26 (2015), 1550033, 5. 

\bibitem{[S3]} M. Slapar, \emph{On complex points of codimension 2 submanifolds}, J. Geom. Anal. 26 (2016), 206 - 219.

\bibitem{[Thom]} R. Thom, \emph{Un lemme sur les applications diff\'erentiables}, Bol. Soc. Math. Mexicana 2 (1956), 59 - 71.

\bibitem{[T]} R. Torres, \emph{CR regular embeddings and immersions of 6-manifolds into complex 4-space}, Proc. Amer. Math. Soc. 144 (2016), 3493 - 3498. 

\bibitem{[W]} C. T. C. Wall, \emph{Classification problems in differential topology. V. On certain 6-manifolds}, Invent. Math. 1 (1966), 355 - 374.

\bibitem{[Whitehead]} G. W. Whitehead, \emph{Elements of homotopy theory}, Graduate Texts in Mathematics, 61. Springer-Verlag, New York-Berlin, 1978. xxi + 744pp.




\end{thebibliography}
\end{document}